\documentclass[a4paper,12pt]{amsart}
\usepackage{graphicx}
\usepackage{amssymb,amsmath,amsthm,graphics}
\usepackage{graphicx}

\newcommand{\arcs}[1]{\stackrel{\rotatebox{90}{\negthinspace\hspace{1pt}\big)}}{#1}} 
\newtheorem{thm}{Theorem}
\newtheorem{definition}[thm]{Definition}
\newtheorem{prop}[thm]{Proposition}
\newtheorem{lemma}[thm]{Lemma}
\newtheorem{cor}[thm]{Corollary}
\newtheorem{exam}[thm]{Example}
\def\ds{\displaystyle}
\def\R{\mathbb R}

\def\om{\omega}
\def\d{\delta}
\def\eps{\varepsilon}

\title{Lipschitzness of the width and diameter functions of convex bodies in $\R^n$}

\author{Oleg Mushkarov}
\address{O. Mushkarov\\Institute of Mathematics and Informatics\\Bulgarian Academy
of Sciences\\Acad. G. Bonchev 8, 1113 Sofia, Bulgaria}

\email{muskarov@math.bas.bg}

\author{Nikolai Nikolov}
\address{N. Nikolov\\Institute of Mathematics and Informatics\\Bulgarian Academy
of Sciences\\Acad. G. Bonchev 8, 1113 Sofia, Bulgaria
\vspace{1mm}
\newline Faculty of Information Sciences\\
State University of Library Studies and Information Technologies\\
Shipchenski prohod 69A, 1574 Sofia,
Bulgaria}

\email{nik@math.bas.bg}

\author{Pascal J. Thomas}
\address{P.J. Thomas\\
Institut de Math\'ematiques de Toulouse; UMR5219 \\
Universit\'e de Toulouse; CNRS \\
UPS, F-31062 Toulouse Cedex 9, France}

\email{pascal.thomas@math.univ-toulouse.fr}

\thanks{The first and the second named authors were partially supported by the Bulgarian National
Science Fund, Ministry of Education and Science of Bulgaria under contract KP-06-N52/3.}


\begin{document}

\begin{abstract} Lipschitz constants for the width and diameter functions of a convex body in
$\R^n$ are found in terms of its diameter and thickness (maximum
and minimum of both functions). Also, a dual approach to thickness
is proposed.
\end{abstract}

\maketitle

\section{Basic definitions and facts}

Let $K$ be a convex body in $\R^n$ ($n\ge 2$), i.e. a convex
compactum with nonempty interior. As is well-known, $\R^n\setminus
K$ is a union of (affine) half-spaces, and any hyperplane which
intersects $K$ without intersecting its interior is called a
\emph{supporting hyperplane}. It is interesting to measure the
width of a convex body by looking at how much two parallel
supporting hyperplanes must be distant from each other. This
depends continuously on the direction of those hyperplanes. The
present paper aims at improving the results about the
variation of the width, and of analogous quantities. First we give
more precise definitions.

\begin{definition}
For any vector $u\in S^{n-1}$ (the unit sphere in $\R^n$),
\begin{itemize}
\item
the \emph{width} of $K$ in direction $u,$
denoted $w_K(u)$, is the distance between
the two supporting hyperplanes of $K$ which are orthogonal to $u;$
\item
the \emph{diameter} of $K$ in direction $u,$ denoted $d_K(u)$, is the maximal length of a
chord of $K$ parallel to $u.$
\end{itemize}
\end{definition}

Note that $w_K$ and $d_K$ are continuous functions on $S^{n-1},$
$w_K\ge d_K,$ and $\max w_K=\max d_K=\d_K:=\mbox{diam }K.$

The number $\om_K=\min w_K$ is called the \emph{thickness} of $K.$
It turns out that $\om_K=\min d_K.$

We point out that $|AB|=d_K(\overrightarrow{AB}/|AB|)$ ($A,
B\in\partial K$) exactly when $K$ admits parallel supporting
hyperplanes at $A$ and $B.$ Such a chord $[AB]$ is called
\emph{diametral}. It is called \emph{double normal} if the
supporting hyperplanes can be chosen orthogonal to $AB.$ If
$|AB|=d_K(\overrightarrow{AB}/|AB|)=\om_K,$ then $[AB]$ is a
double normal chord.

The facts listed above  can be found e.g.~in \cite{Sol} (see also \cite{MW} about $w_K$ and $\om_K$).
\smallskip

It is known that $w_K$ is a Lipschitz function w.r.t. the Euclidean
norm $||\cdot||$ with constant $\d_K$ (see
\cite[Proposition]{MW}). This constant does not take into account
the fact that for some convex bodies, the width can
vary much less---for instance, there are convex bodies
of constant width, the obvious example being
the Euclidean ball.

In Section \ref{lipcst}, we will prove that $w_K$  and $d_K$ are
Lipschitz functions with constants
$$M_K=\sqrt{\d_K^2-\om_K^2}\mbox{\ \ resp.\ \ }N_K=\frac{\d_K}{\om_K}\sqrt{\d_K^2-\om_K^2},\mbox{\ }$$
(which vanish for bodies of constant width)
w.r.t.~$\rho(u,v)=\arccos|\langle u, v\rangle|$,
the spherical distance on the projective space  $S^{n-1}/\pm I$, which is a pseudo-distance on $S^{n-1}$. Note that $\rho(u,v)\le\frac{\pi}{2\sqrt2}||u\pm v||.$
Those results are refined in Section   \ref{refine}.

In Section \ref{pole}, we study yet another notion of width;
while we previously were looking at chords
parallel to each other, so passing through the same point at
infinity in the projective space, now we consider chords all
passing through a given  point in affine space.
This leads to some interesting continuity phenomena.

\section{Lipschitzness of $w_K$ and $d_K$}
\label{lipcst}

For $u,v \in S^{n-1}, u\neq \pm v $, denote
\[
\Delta w_K(u,v) := \frac{|w_K(u)-w_K(v)|}{\rho(u,v)}, \quad
\Delta d_K(u,v) := \frac{|d_K(u)-d_K(v)|}{\rho(u,v)}.
\]

\begin{thm}
\label{main} For any convex body $K$ in $\R^n$ and $u, v\in
S^{n-1}$:
\begin{enumerate}
\item $\Delta w_K(u,v)\le M_K$;
\item $\Delta d_K(u,v)\le N_K$.
\end{enumerate}
\end{thm}

\begin{proof}

Since $K$ is the intersection of convex polytopes $K_1\supset K_2\supset\dots$
and $w_{K_i}\downarrow w_K,$ $d_{K_i}\downarrow d_K,$ $\d_{K_i}\downarrow\d_K,$ and
$\om_{K_i}\downarrow\om_K,$ we may assume that $K$ is a convex polytope.
\smallskip

{\it Proof of Part (1).} Since $w_K$ is already known to be
continuous, it will be enough to prove the inequality outside of a
set of dimension $n-2$ on the unit sphere. Let $E$ be an edge of
$K$ (non-trivial line segment contained in $\partial K$ such that
its relative interior is not contained in the relative interior of
any face of larger dimension), and $H_E$ the  hyperplane of
vectors orthogonal to $E$. For further reference, call those
\emph{exceptional hyperplanes}. Let $u\in \mathcal W_K := S^{n-1}
\setminus \bigcup_{E \mbox{ \small edge of }K} H_E$.  Then $u$ is
not orthogonal to any face of $\partial K$ of any dimension $d$,
$1\le d\le n-1$. For such a vector $u$, each of the two supporting
hyperplanes $H_1$ and $H_2$ orthogonal to $u$ intersects $K$ at a
single point, denoted respectively by $A$ and $B$.

\begin{lemma}
\label{limw}
For any $u \in \mathcal W_K$ and $A, B$ chosen as above,
\[
w_K'(u):= \limsup_{u',u''\to u} \Delta w_K(u',u'')
= \sqrt{|AB|^2-w^2_K(u)} \le M_K.
\]
\end{lemma}
\begin{proof}
For any $v$ close enough to $u$, the corresponding supporting
hyperplanes intersect $K$ at the same points $A$ and $B$, so that
$w_K(v)= |\langle AB, v\rangle|$ and
\[
\frac{|w_K(u')-w_K(u'')|}{\|u'-u''\|} = \left| \left\langle AB, \frac{u'-u''}{\|u'-u''\|}\right\rangle\right|.
\]
Note that $\langle u'-u'', (u'+u'')/2 \rangle=0$ and that $(u'+u'')/2$ tends to $u$.
Passing to a subsequence, we may assume that $\frac{u'-u''}{\|u'-u''\|} \to \tilde u$, with
$\tilde u \in S^{n-1}$, $\langle \tilde u, u \rangle=0$.
Let $\tilde A$ be the
unique point in  the supporting hyperplane going through $B$
such that $A\tilde A$ is parallel to $u$; then $\langle AB, \tilde u\rangle = \langle \tilde AB, \tilde u\rangle$.
So, observing that $\rho(u',u'') \sim \|u'-u''\|$ as $\|u'-u''\|\to0$,
\begin{multline*}
\lim_{u',u''\to u} \Delta w_K(u',u'') =
\lim_{u',u''\to u}\left| \left\langle AB, \frac{u'-u''}{\|u'-u''\|}\right\rangle\right|
\\
= \langle \tilde AB, \tilde u\rangle \le |\tilde AB| = \sqrt{|AB|^2-|A\tilde A|^2} =
\sqrt{|AB|^2-w^2_K(u)}.
\end{multline*}
We can obtain the case of equality above if we choose $u'=u$ and $u''$ so that $\tilde u$
is parallel to $(\tilde A B)$, which shows that the limes superior is equal to the righthand
side.
\end{proof}

Any two vectors $u,v$ in the same connected component of
$\mathcal W_K$ can be joined
within the component by a geodesic (arc of a great circle). If
the inequality in Theorem \ref{main} (1) was violated, by dichotomy, we could find
arbitrarily small arcs $\arcs{u'u''}$
within the geodesic segment where the inequality
would be violated, but since $\rho(u',u'')  \sim \|u'-u''\|$ as both quantities tend
to $0$, this would violate Lemma \ref{limw}.

If $u$ or $v$ lies on an exceptional hyperplane, and the other in
an adjacent component, we obtain the same result by continuity. If
$u$ and $v$ belong to different components, the shorter of the two
arcs on the great circle containing them will intersect the
exceptional hyperplanes in a finite number of points, and we
conclude again by continuity.  If $u$ and $v$ are on distinct
exceptional hyperplanes, the same method works.  Finally, if $u$
and $v$ lie on the same exceptional hyperplane, we may perturb one
of them slightly, apply the previous estimate, and conclude again
by continuity.

\smallskip

{\it Proof of Part (2).} As in Part (1), we exclude a finite set
of exceptional vector hyperplanes from the unit sphere. This time
those are the hyperplanes $H_D$ generated by a set $D$ of vertices
of $K$ with $\# D=n$. When $u \in \mathcal R_K := S^{n-1}
\setminus \bigcup_{D\subset V(K), \# D=n} H_D $, we can choose a
diametral chord for $u$ passing through a vertex $A$ of $K$; its
other extremity $B$ will have to be in the interior of an
$(n-1)$-dimensional face of $K$. Denote by $\tilde d_K(u)$ the
distance between the parallel supporting hyperplanes of $K$, $H_1$
and $H_2$, that contain $A$ and $B$ respectively.

\begin{lemma}
\label{limd}
For any $u\in \mathcal R_K$,
\[
d_K'(u):= \limsup_{u',u''\to u} \Delta d_K(u',u'')
=
\frac{d_K(u)}{\tilde d_K(u)}\sqrt{d^2_K(u)-\tilde d^2_K(u)}\le N_K.
\]
\end{lemma}

\begin{proof}
There is a neighborhood of $u$, $\mathcal N\subset  \mathcal R_K$,
such that for any $t\in \mathcal N$, a diametral chord for $t$
also passes through $A$ and the face containing $B$. Assume $u',
u''\in \mathcal N$.

Let $H$ be the \emph{strip} of $\mathbb R^n$ determined by $H_1$
and $H_2$ (that is, the intersection between two distinct half
spaces determined by $H_1$ or $H_2$ that contains $K$). Consider
the points $B' \in H_2$ (resp. $B''\in H_2$) such that $AB'$(resp.
$AB''$) is parallel to $u'$ (resp. $u''$), then $d_K(u')=|AB'|$,
$d_K(u'')=|AB''|$. Assume as we may that $|AB''|\ge |AB'|$. In
$\triangle AB'B''$, let $h$ be the distance from $A$ to the line
$(B'B'')$ and  $\alpha, \beta', \beta''$ be the angles at the
vertices $A, B', B''$ respectively. Again we use that
$\rho(u',u'') \sim \|u'-u''\|$. Then by the law of sines and
$\alpha+ \beta'+ \beta''=\pi$,
\begin{multline*}
\frac{|AB''|- |AB'|}{\|u'-u''\|} = \frac{|AB'|(\sin \beta' - \sin \beta'')}
{ (\sin \beta'')( 2 \sin \alpha/2) } = |AB'| \frac{\sin \frac{\beta'-\beta''}
2  \cos \frac{\beta'+\beta''}2}{ \sin \beta''  \sin \alpha/2 }
\\
= |AB'| \frac{\cos (\beta''+  \frac{\alpha}2)}{ \sin \beta'' } \le |AB'| \cot \beta'' =
|AB'| \sqrt{\frac{|AB''|^2}{h^2} -1 } \le  |AB'| \sqrt{\frac{|AB''|^2}{\tilde d_K(u)^2} -1 },
\end{multline*}
since $h\ge \tilde d_K(u)$. Passing to the limit, we obtain that $N_K$
is an upper bound for the limes superior we are considering.

Observe that as $u', u''\to u$, then $\alpha \to 0$, $|AB'|,|AB''|\to d_K(A)$,
but when the unit vectors parallel to $(B'B'')$ admit a limit $v$,
$h\to \mbox{dist}(A, L)$
where $L$ is the line through $B$ parallel to $v$.
If we choose $A' \in H_B$ such that $|AA'|= \tilde d_K(u)$, and $u',u''$
so that $B', B''$ tend to $B$ along the line $(BA')$, then
$h\to  \tilde d_K(u)$, and $$\lim_{u',u''\to u}\Delta d_K(u',u'')=
\frac{d_K(u)}{\tilde d_K(u)}\sqrt{d^2_K(u)-\tilde d^2_K(u)}.$$
\end{proof}

We finish the proof of Part (2) in the same way as that of Part
(1), substituting Lemma \ref{limd} for Lemma \ref{limw}.
\end{proof}

\noindent{\bf Remark.} Notice that to obtain Theorem \ref{main}, it is enough to
use the upper bound for the derivative in Lemmas \ref{limd} and \ref{limw}.
\smallskip

\noindent{\bf Open question.} Can one replace $\rho(u,v)$ by $||u-
v||$ in Theorem \ref{main}?
\smallskip

The constants we have obtained are the
best possible in the following sense.

\begin{thm}
Let $\delta\ge \omega >0$, and $\mathcal K_{\delta,\omega}:= \{
K\subset \R^n \mbox{ convex body such that } \delta_K=\delta
\omega_K=\omega\}$. Then
\[
\max_{K \in \mathcal K_{\delta,\omega} } \sup_{u\neq \pm v \in S^{n-1} }
\Delta w_K(u,v)= \sqrt{\d^2-\om^2} ,
\quad
\max_{K \in \mathcal K_{\delta,\omega} } \sup_{u\neq \pm v \in S^{n-1} }
\Delta d_K(u,v)= \frac\delta\omega \sqrt{\d^2-\om^2} .
\]
\end{thm}

\begin{proof}
That the maximum in each equality is bounded above by the righthand side follows
immediately from Theorem \ref{main}.

To prove the converse inequality, consider the following example.
By scaling, it is enough to assume $\delta=2$. For $\omega\in
(0,2]$, let $K:= \overline B(0,1) \cap \{-\omega/2 \le x_1 \le
\omega/2\}$. Clearly $\omega_K=2$. Since $\pi_{\R^2\times\{0\}}
(K)\subset  K \cap (\R^2\times\{0\})$, and since the tangent space
to $\partial B(0,1)$ at any point $u_\theta:= (\cos \theta, \sin
\theta, 0, \dots, 0)$ contains $\{(0,0)\}\times\R^{n-2}$, it is
enough to work on the case $n=2$.

Let $\omega=2\cos \alpha$, $\alpha \in [0, \pi/2)$. Then it is
easy to see that for $\pi/2 \ge \theta\ge \alpha$,
$w_K(u_\theta)=2$, and for $0\le \theta\le \alpha$,
$w_K(u_\theta)=2 \cos (\alpha - \theta)$. Therefore $w_K= 2\cos
\alpha=\omega$, and $w'_K(u_0)= 2 \sin \alpha =
\sqrt{4-\omega^2}$, which proves the first part of the theorem.

On the other hand, for $\pi/2 \ge \theta\ge \alpha$, $d_K(u_\theta)=2$,
and for $0\le \theta\le \alpha$, $d_K(u_\theta)=\frac{\omega}{\cos \theta}$.
So (with analogous notation) $d'_K(u_\alpha)= \omega\frac{\sin \alpha}{\cos^2 \alpha}=
\frac2\omega \sqrt{4-\omega^2}$.
\end{proof}

\section{Dual approach to $\om_K$}
\label{pole}

One may define a function $e_K$
in a dual way to $d_K,$ replacing vectors by points. More precisely, for any point $O$
in $\R^n$ set $e_K(O)$ to be the diameter of $K$ w.r.t.~$O,$ i.e.~the maximal length of
a chord of $K$ cut by a line through $O.$ It is clear that
$\d_K=\ds\max_{O\in \R^n} e_K(O)=\ds\limsup_{O\to\infty} e_K(O).$

\begin{prop}
The function $e_K$ is  upper semicontinuous  on $\R^n$ and its
restrictions to $K^\circ$ and $\R^n\setminus K$ are continuous.
\end{prop}

\begin{proof}
Let $l(\Delta)$ denote the length of a segment $\Delta.$ Suppose
to get a contradiction that upper semicontinuity fails at a point
$O$, and let $L$ be a line realizing the maximum in $e_K$. Then
there are $\eps>0$, a sequence of points $O_k$ and unit vectors
$u_k$ such that if we write $L_k:= O_k + \R u_k$, then $l(L_k \cap
K) \ge l(L \cap K)+\eps$.  Passing to a subsequence, we can assume
that $L_k \cap K$ converges in the Hausdorff sense to a line
segment $S$ contained in $L'\cap K$ for some line $L'$ with $O\in
L'$. Then $l(L' \cap K) \ge l(L \cap K)+\eps$, which contradicts
the maximality of $L$.

If $O\notin K$, choose a line $L_O$ such that the maximum chord
length is attained.  Since $K$ has non empty interior and is
convex, we can find a line segment arbitrarily close to $L_O\cap
K$, such that a $\delta$-neighborhood of it, $U$, is contained in
$K$. Taking $O'$ close enough to $O$, which has a positive
distance from $K$, we can obtain a line $L'$ so that $l(L'\cap
U)\ge l(L_O\cap K)-\eps$.

If $O\in K^\circ$, a neighborhood of $L_O\cap K$ is contained in
$K$, and for $O'$ close enough to $O$ and $L'$ passing through
$O'$ and parallel to $L_O$, $l(L'\cap K)\ge \eps$.
\end{proof}

Recall that a point $O\in K$ is said to be \emph{extreme} if for
any line segment $S\subset K$ such that $O\in S$, then $O$ is an
endpoint of $S$.

\begin{prop} The function $e_K$ is continuous
at $O\in\partial K$ if and only if there exist a chord $S$ of maximal length so that
$O$ is an endpoint of $S$; equivalently, if there exists $O'\in \partial K$ such that
$e_K(O)=|OO'|$.

In the case of continuity, $O'$ must be an extreme point of $K$.
\end{prop}

It follows in particular that $e_K$ is continuous at any extreme point of $K$.

\begin{proof}
First, suppose that there exist a chord $S$ of maximal length so
that $O$ is an endpoint of $[OO']:=S$. Then given $\eps>0$, there
exists a ball $B(A_1,\eps_1) \subset K\cap B(O',\eps/2)$, with
$\eps_1>0$. So the convex hull $C$ of $\{O\} \cup B(A_1,\eps_1)$
is contained in $K$. For $\eps_2>0$ small enough and
$|OP|<\eps_2$, $l((A_1P)\cap C) \ge |OA_1|-\eps/2 \ge
|OO'|-\eps$, q.e.d.

If $O'\in S' \subset \partial K$ was not extreme, looking at a point $A\in S'$
on one or the other side of $O'$, we could obtain $|OA|>|OO'|=e_K(O)$, thus violating
the maximality assumption.

Conversely, assume that for any chord $S:=[O'O'']$  containing $O$
such that $e_K(O)= l(S)= |OO'|+|OO''|$ one has $O\neq O'$, $O\neq
O''$. Without loss of generality assume that $|OO'|\le |OO''|$. We
set $r:=\frac12 \min |OO'| \le \frac14 e_K(O)$, where the minimum
is taken over all chords $S$ as above. Then by compactness there
exists a $\delta
>0$ such that
\[
\sup \{ |OA| : A \in K, |AO'|\ge r, |AO''|\ge r, \forall O', O''\} \le e_K(O)-\delta.
\]

Let $H$ be a supporting hyperplane for $K$ at $O$, and let $H^+$ be the connected
component of $\R^n\setminus H$ that does not meet $K$. Let
$P \in H^+$, $A\in K$.
For $|OP|\le \eps$, if $A$ satisfies $|AO'|\ge r, |AO''|\ge r, \forall O', O''$,
endpoints of chords of maximal length containing $O$,
then $l((AP)\cap K) \le |AP| \le |OA|+\eps \le e_K(O)-\delta +\eps \le e_K(O)-\delta/2$
for $\eps$ small enough.
If on the other hand there exists $O'$ or $O''$ such that $|AO'|, |AO''| \le r$, then
(taking $O'$ and $O''$ to be both endpoints of the same chord)
$$
l((AP)\cap K) \le |AP| \le |OA|+\eps \le |OO''| + \frac12 |OO'| +\eps
\le e_K(O) - r +\eps \le e_K(O) - r/2,
$$
for $\eps$ small enough. So $\liminf_{P\to O, P \in H^+} e_K(P) < e_K(O)$.
\end{proof}
\begin{exam} Let $K=\triangle ABC$. Then the set of points of continuity
of $e_K$ on $\partial K$ reduces to $\{A,B,C\}$ if and only if $\triangle ABC$
is equilateral.
\end{exam}

Set now
$\eps_K=\inf e_K.$ Observe that this infimum may fail to be
attained. For example, if $K$ is a parallelogram with heights
$h_1\le h_2,$ then $h_1=\om_K=\eps_K<e_K(O)$ for any point $O.$

On the other hand, we have the following

\begin{exam} Let $K=\triangle ABC,$ with heights $h_a\le h_b, h_c.$ Then $\eps_K=\om_K=h_a.$
Moreover, $e_K(O)=\eps_K$ for any point $O$ on the line $(AA_1)\perp (BC)$ ($A_1\in[BC]$)
such that $\angle ABO\ge 90^\circ$ and $\angle ACO\ge 90^\circ.$
\end{exam}

\begin{proof} Note that for any point $O'$ there is a vertex $D$ of $K$ such that the
line $(O'D)$ meets its opposite side at a point $D'$. Then
$e_K(O')\ge|DD'|\ge h_d\ge h_a=\om_K.$

It remains to prove that $e_K(O)=h_a$ for $O$ as above. It is
enough to show that if $E\in]AB[,$ $E_1\in]BA_1[,$ and
$O\in(EE_1),$ then $|EE_1|<|AA_1|.$ The last follows by
$|OE_1|>|OA_1|$ and $|OE|<|OA|$ (since $\angle AEO>\angle ABO\ge
90^\circ$).
\end{proof}

The observations above suggest the following

\begin{prop} For any convex body $K$ one has that
$\om_K=\eps_k=\tilde\eps_K:=\ds\liminf_{O\to\infty}e_K(O).$
\end{prop}

\begin{proof} Note that if $[AB]$ is a diametral chord, then
$\ds\lim_{AB\ni O\to\infty}e_K(O)=|AB|$ and hence $\eps_K\le\tilde\eps_K\le\om_K.$
On the other hand, for any point $O$ there is a diametral chord $[AB]$ such that
$O\in AB$ (see e.g.~\cite[6.6]{Sol}) and thus $\eps_K\ge\min d_K=\om_K.$
\end{proof}

\section{Refinement of the Lipschitz constants}
\label{refine}

\subsection{Variations of $w_K$.}
\label{varw}

Recall that in order to obtain Lemma \ref{limw}, we needed to assume $u\in \mathcal W_K$.
For $K$ a convex polytope, we will establish an analogous result for all $u\in S^{n-1}$.
For any convex body $K$, we also derive a more precise upper
bound for $w_K'(u)$ which depends on $u$.

As before, for a convex body $K$, denote by $H_1(u)$ and $H_2(u)$ the two supporting hyperplanes
of $K$ orthogonal to $u$, and set
\[
s_K(u):=\max\{|A'B'|: A'\in K\cap H_1(u), B'\in K\cap H_2(u)\},\quad
p_K(u):= \sqrt{ s^2_K(u)-w^2_K(u)}.
\]

\begin{lemma}
\label{pusc} Let $K_j$ be a decreasing sequence of convex bodies
with limit $K$ and let $u_j \in S^{n-1}$ be a sequence with
$\lim_j u_j = u$. Then $\limsup_j s_{K_j} (u_j) \le s_K(u)$.

As a consequence, the analogous inequality holds for $p_K$,
and both $s_K$ and $p_K$ are upper semicontinuous on $S^{n-1}$.
\end{lemma}

\begin{proof}
For any $x \in S^{n-1}$,
denote by $H^j_1(x)$ and $H^j_2(x)$ the two supporting hyperplanes of $K_j$ orthogonal to $x$.

We choose $A_j \in K\cap H^j_1(u_j)$, $B_j \in K\cap H^j_2(u_j)$
such that $s_{K_j} (u_j) = |A_j B_j|$.

Since $s_{K_j}$ is a bounded sequence, passing to a subsequence
$K'_j$ we may assume that $s_{K'_j}$ converges to $\limsup_j
s_{K_j}$, and again passing to a subsequence, that $A_j\to A^* \in
H_1(u)\cap K$, $B_j\to B^* \in H_2(u)\cap K$. Then $s_K(u)\ge
|A^*B^*| =\lim_j |A_jB_j| = \limsup_j s_{K_j}$.

To obtain the upper semicontinuity, it is enough to specialize
to the case where $K_j=K$ for all $j$.
\end{proof}

For $K$ any convex body, define
$\hat M_K:=\max_{u\in S^{n-1}}p_K(u)$.
Since $p_K$ is upper semicontinuous,  the maximum in the
definition is indeed attained.

For $K$  a convex polytope, define
$ \tilde M_K:=\sup_{u\in \mathcal W_K}p_K(u)$.

If $K$ is a convex polytope and $u \in \mathcal W_K$, $K\cap H_1(u)$ and $K\cap H_2(u)$
are singletons, and Lemma  \ref{limw} says that $w'(u)=p_K(u)$.

\begin{prop}
\label{nongen}
For $K$ a convex polytope, $ \tilde M_K=\hat M_K$ and
for any $u\in S^{n-1}$, $w_K'(u)= p_K(u) \le \hat M_K$.
\end{prop}

\begin{proof}
To prove that $ \tilde M_K=\hat M_K$ ,
it will be enough to show :
\[
 \forall u\in S^{n-1}\setminus \mathcal W_K,
\limsup_{u'\to u, u'\neq u} p_K(u')= p_K(u).
\]
That the lefthand side is less or equal to $p_K(u)$
follows from Lemma \ref{pusc}.

To obtain the reverse inequality, let $A_m, B_m$ be two points
where the maximum in the definition of $s_K(u)$ is achieved. Let
$\mathcal P$ be the affine plane generated by $A_m, u$, and
$\overrightarrow{A_m B_m} $, $\pi_{\mathcal P}$ the orthogonal
projection to $\mathcal P$, and $\mathcal L$ the
$(n-2)$-dimensional vector space orthogonal to $\mathcal P$.

Then both $H_1(u)$ and $H_2(u)$ are parallel to $\mathcal L \oplus
v$, where $v \in \mbox{Span }(u,\overrightarrow{A_m B_m} )$,
$\|v\|=1$, $v\perp u$. Furthermore we can choose $ \langle
\overrightarrow{A_m B_m}, v \rangle \ge 0$. Notice that
$\pi_{\mathcal P}(H_1(u)) = H_1(u)\cap \mathcal P=A_m+\R v :=L_A$;
$\pi_{\mathcal P}(H_2(u)) = H_2(u)\cap \mathcal P=B_m+\R v :=L_B$,
and because we have supporting hyperplanes, $\pi_{\mathcal P}(K)$
is contained in the closed strip between $L_A$ and $L_B$. The line
segment $\pi_{\mathcal P}(H_1(u)\cap K)\subset L_A$, and by our
maximality hypothesis, for any $t<0$, $A_m+tv \notin \pi_{\mathcal
P}(H_1(u)\cap K)$. Similarly, $\pi_{\mathcal P}(H_2(u)\cap
K)\subset L_B$, and
 for any $t>0$,
$B_m+tv \notin \pi_{\mathcal P}(H_2(u)\cap K)$.

Since $\pi_{\mathcal P}(K)$ is a closed polygonal convex body,
if we orient the plane so that $\angle (\overrightarrow{A_m B_m}, v) >0$,
then there exists $\eps_0>0$ such that for any unit vector $v' \in \mbox{Span }
(u,\overrightarrow{A_m B_m} )$ with $0< \angle (v,v') < \eps_0$,
then $A_m+\R v'$ is a supporting line for $\pi_{\mathcal P}(K)$,
with $A_m+\R v' \cap \pi_{\mathcal P}(K) = \{ A_m\}$,
and  $B_m+\R v'$ satisfies the analogous properties.
Therefore $H_1':=\pi_{\mathcal P}^{-1} (A_m+\R v')$
and $H_2':=\pi_{\mathcal P}^{-1} (B_m+\R v')$ are
parallel supporting hyperplanes for $K$, orthogonal to a vector $u'$
which can be made arbitrarily close to $u$, such that
$H_1' \cap K \subset A_m + \mathcal L$ and
$H_2' \cap K \subset B_m + \mathcal L$.

Again by our maximality hypothesis, we have that
\[
K \cap (A_m + \mathcal L)
= \{ A_m\} , K \cap (B_m + \mathcal L)= \{ B_m\},
\]
so finally the supporting hyperplanes $H_1'$ and $H_2'$ intersect
$K$ only at $A_m$ and $B_m$ respectively, so $u'\in \mathcal W_K$
and $s_K(u')=s_K(u)$; it is easy to see that $s_K(u')\to s_K(u)$
as $u'\to u$, and we are done with the first part of the
proposition.

Now consider again $u_0 \in S^{n-1}\setminus \mathcal W_K$.
Taking $u \in \mathcal W_K$ close to $u_0$ such that $p_K(u)$
is close to $p_K(u_0)$, and $u',u''$ which tend to $u$ and
come close to the limes superior in Lemma \ref{limw}, we see
that $w'(u_0) \ge p_K(u_0)$ (perform a diagonal process
to have $u'$ and $u''$ tending to $u_0$).

To obtain the reverse inequality, for any $\eps>0$, by the upper semicontinuity
of $p_K$ and the first part of the proposition, there is a neighborhood $U$
of $u_0$ such that for any $u \in U$, $w'(u_1)\le p_K(u_0)+\eps$. If $u', u''\in U$,
we can connect them by an arc $\gamma$ of a great circle and reasoning as in the proof of
Lemma \ref{limw},
\[
|w_K(u')-w_K(u'')|\le \rho(u',u'') \sup_{u\in \gamma} w'(u)
\le (1+\eps) \|u'-u''\| (p_K(u_0)+\eps) ,
\]
reducing $U$ if needed. This proves that $w'(u_0) \le p_K(u_0)$.
\end{proof}

As a consequence of Theorem \ref{main} and the proof above, we obtain the following.
\begin{cor}
\label{global}
For any convex polytope $K\subset \R^n$,
\[
\sup_{u\neq \pm v \in S^{n-1} } \Delta w_K(u,v)
= \tilde M_K=\hat M_K \le M_K.
\]
\end{cor}

\begin{prop}
\label{bodies}
For any convex body $K\subset \R^n$ one has that
$$
\sup_{u\neq \pm v \in S^{n-1} } \Delta w_K(u,v) \le \hat M_K\mbox{\ \ and\ \ }w_K'(u)\le p_K(u).
$$
\end{prop}

\begin{proof}
We will say that  $\mathcal A \subset S^{n-1}$ is
convex whenever it is the intersection of a convex cone with $S^{n-1}$,
i.e.  $\R_+ \mathcal A$ is convex, i.e. for any two points $a,b \in \mathcal A$,
the geodesic segment (arc of a great circle) from $a$ to $b$ is contained in $\mathcal A$.

Take a sequence of polytopes $K_j$ decreasing to $K$,
and a closed convex subset $\mathcal N \subset S^{n-1}$ such that
for any $u , v \in
\mathcal N$, $\rho(u,v)\le \pi/2$. We claim that
\[
\sup_{u\neq  v \in \mathcal N} \Delta w_K(u,v)\le \max_{\mathcal N} p_{K}.
\]
Indeed, by the proof of Theorem \ref{main} (1), carried out
on geodesics remaining within $\mathcal N$,
\[
| w_{K_j} (u) - w_{K_j} (v) | \le \rho(u,v) \max_{\mathcal N} p_{K_j} .
\]
We claim that $\limsup_j \max_{\mathcal N} p_{K_j} \le \max_{\mathcal N} p_{K}$.
Indeed, take a sequence $(u_j)_j \subset \mathcal N$ such that
$\lim_j p_{K_j}(u_j)= \limsup_j \max_{\mathcal N} p_{K_j}$, and
$u_j\to u \in \mathcal{N}$, then Lemma \ref{pusc} implies that
\[
\lim_j p_{K_j}(u_j) \le p_{K}(u) \le \max_{\mathcal N} p_{K},
\]
and the claim is proved.

To get the statement over the whole sphere, recall that $w_K(u)=w_K(-u)$, and notice that
for any $u,v \in S^{n-1}$, $\rho(u,v)\le \pi/2$ or $\rho(u,-v)\le \pi/2$,
and we can choose a convex set $\mathcal N$ accordingly.

To get the statement about $w_K'(u)$, take a sequence of convex closed
neighborhoods of $u$ in $S^{n-1}$ converging to $\{u\}$,
and apply the upper semicontinuity of $p$.

Notice that we could also deduce the statement about $\Delta w_K$ from
the one about $w_K'$.
\end{proof}

\noindent{\bf Open question.}
Can the  inequalities in  Proposition \ref{bodies} be replaced
by equalities for all convex bodies?

\subsection{Variations of $d_K$.}

In analogy to the beginning of Subsection \ref{varw}, we extend the notations
defined before Lemma \ref{limd}. Given $u\in S^{n-1}$, for any diametral
chord $[AB]$ we have $|AB|=d_K(u)$ and there exists some parallel
supporting hyperplanes for $K$ at $A$ and $B$. We define
\begin{multline*}
r_K (u) :=
\\
\inf \left\{ \mbox{dist }(H_1, H_2): [AB] \mbox{ diametral chord and }
H_1\ni A, H_2\ni B \mbox{ supporting hyperplanes}
\right\},
\end{multline*}
and
\[
q_K(u):=
d_K(u)\sqrt{\frac{d^2_K(u)}{r^2_K(u)}-1}.
\]
For $K$ a convex polytope and $u\in \mathcal R_K$, Lemma \ref{limd} says that $d_K'(u)= q_K(u)$.
\smallskip

\noindent{\bf Open question.}
For $K$  a convex polytope, do we have
$\sup_{u\in \mathcal R_K}q_K(u) = \max_{u\in S^{n-1}}q_K(u)$?
\smallskip

Even without an answer to that question, we can recover a result analogous to Proposition \ref{bodies}.

\begin{lemma}
\label{qlsc}
Let $K_j$ be a decreasing sequence of convex bodies
with limit $K$ and let $u_j \in S^{n-1}$ be a sequence with
$\lim_j u_j = u$. Then $\liminf_j r_{K_j} (u_j) \ge r_K(u)$.

As a consequence, the reverse inequality holds for $q_K$,
$r_K$ is lower semicontinuous and $q_K$ upper semicontinuous on $S^{n-1}$.
\end{lemma}

\begin{proof}
We begin by proving that $\lim_j d_{K_j} (u_j)= d_{K} (u)$.  Indeed, since $K\subset K_j$
for all $j$, $d_{K_j} (u_j) \ge d_{K} (u_j) \to d_K(u)$  since $d_K$
is continuous. On the other hand, choosing $A_j$, $B_j$ such that $d_{K_j} (u_j) = |A_j B_j|$,
any convergent subsequences will tend respectively to $A^*$, $B^*$ such that
$(A^*B^*)$ is parallel to $u$, so we deduce that $\limsup_j d_{K_j} (u_j) \le |A^*B^*|\le
d_K(u)$ and we are done.

Considering  a subsequence such that $d_{K_j} (u_j) = |A_j B_j|\to d_K(u)$, we have
$A_j \in K\cap H^j_1$, $B_j \in K\cap H^j_2$ with $\mbox{dist }(H^j_1,H^j_2)= r_{K_j} (u_j)$.
We can choose a further subsequence
such that $r_{K_j} (u_j)\to \liminf_j r_{K_j} (u_j)$, and $H^j_1(u_j)$,
$H^j_2(u_j)$ converge to hyperplanes $H_1$, $H_2$. Then $H_1$ and $H_2$ are
supporting hyperplanes for $K$ at $A^*$ and $B^*$ respectively, and $[A^*;B^*]$
is now a diametral chord for $u$,
so that $\liminf_j r_{K_j} (u_j)=\mbox{dist }(H_1,H_2) \ge r_K(u)$.
\end{proof}

\begin{prop}
\label{dbodies}
For any convex body $K\subset \R^n$ one has that
$$
\sup_{u\neq \pm v \in S^{n-1} } \Delta d_K(u,v) \le \hat N_K\mbox{\ \ and\ \ }d_K'(u)\le q_K(u).
$$
\end{prop}

\begin{proof}
The proof follows the same lines as that of Proposition \ref{bodies} (which
makes no use of Proposition \ref{nongen}), replacing $p_K$ by $q_K$ and $M_K$ by $N_K$.
\end{proof}

\subsection{Examples.}

\begin{exam}
\label{triangle} Let $K=\triangle ABC$, with side lengths $a\le b
\le c$, and heights $h_a, h_b, h_c$.

Then $\d_K=c$, $\omega_K=h_c$, $M_K=\sqrt{c^2-h_c^2}$, $N_K=\frac{c}{h_c}\sqrt{c^2-h_c^2}$.

Let $\hat v$ be a unit vector parallel to the side of length $c$,
$\hat u$ be the unit vector orthogonal to the side of length $b$
such that $\langle \hat u , \hat v\rangle \ge 0$, and $\gamma$ the
geodesic arc between them. Then
$$
\lim_{\gamma\ni u',u''\to \hat u} \Delta w_K(u',u'')= \sup_{u\neq
\pm v \in S^{1} } \Delta w_K(u,v)= \sqrt{c^2-h_b^2}=\hat M_K$$

$$\lim_{\gamma\ni v',v''\to \hat v} \Delta d_K(v',v'')= \sup_{u\neq
\pm v \in S^{1} } \Delta d_K(u,v)=
\frac{c}{h_b}\sqrt{c^2-h_b^2}=\hat N_K.$$

\end{exam}
We see that $\hat M_K=M_K$ and $\hat N_K=N_K$ if and only if $b=c$.

\begin{exam}
\label{box} Let $K=[0,a_1]\times\dots\times[0,a_n],$
$0<a_1\le\dots\le a_n.$ Then $\om_K=a_1$ and
$\delta_K=\sqrt{a_1^2+\dots+a_n^2}.$ Let $u=e_1,$ $v=(a_1e_1+\dots
a_ne_n)/\delta_K,$ and $\gamma$ be the geodesic joining $u$ and
$v$ on $S^{n-1}.$ One may check that
$$
\lim_{\gamma\ni u',u''\to u}\Delta w_K(u',u'')=M_K,\quad
\lim_{\gamma\ni v',v''\to v}\Delta d_K(u',u'')=N_K.
$$
\end{exam}

\begin{exam}
\label{ellipse} Let $\displaystyle K=\left\{\frac{x^2}{a^2}+\frac{y^2}{b^2}=1\right\}.$
Then for $u=(\cos\alpha,\sin\alpha)\in S^1$ we have
$$w_K(u)= 2\sqrt{a^2\cos^2\alpha+b^2\sin^2\alpha},\quad s_K (u)= 2\sqrt
\frac{a^4\cos^2\alpha+b^4\sin^2\alpha}{a^2\cos^2\alpha+b^2\sin^2\alpha},
$$
$$d_K(u)= \frac {2ab}{\sqrt{a^2\sin^2\alpha+b^2\cos^2\alpha}},\quad r_K (u)= 2ab\sqrt
{\frac{a^2\sin^2\alpha+b^2\cos^2\alpha}{a^4\sin^2\alpha+b^4\cos^2\alpha}},
$$
$$p_K(u)=\frac{|(a^2-b^2)\sin2\alpha|}{\sqrt{a^2\cos^2\alpha+b^2\sin^2\alpha}},
\quad q_K(u)=\frac{ab|(a^2-b^2)\sin2\alpha|}{\sqrt{(a^2\cos^2\alpha+b^2\sin^2\alpha)^3}},$$
$$\lim_{u',u''\to u}\Delta w_K(u',u'')=p_K(u),\quad
\lim_{u',u''\to u}\Delta d_K(u',u'')=q_K(u).$$
\end{exam}

\end{document}